\documentclass[a4paper,12pt]{article}
\usepackage{amsfonts,amsthm,amsmath}
\newtheorem{thm}{Theorem}
\newtheorem{prop}[thm]{Proposition}

\theoremstyle{remark}
\newtheorem{rem}[thm]{Remark}
\newcommand{\FF}{\mathbb{F}}
\newcommand{\ZZ}{\mathbb{Z}}
\newcommand{\RR}{\mathbb{R}}
\newcommand{\CC}{\mathbb{C}}

\DeclareMathOperator{\swe}{swe}

\begin{document}

\title{An Upper Bound on the Minimum Weight of Type~II $\ZZ_{2k}$-Codes}

\author{Masaaki Harada\thanks{Department of Mathematical Sciences,
Yamagata University, Yamagata 990--8560, Japan and
PRESTO, Japan Science and Technology Agency, Kawaguchi,
Saitama 332--0012, Japan. 
email: mharada@sci.kj.yamagata-u.ac.jp.
This work was partially supported by JSPS KAKENHI (20540103).}
and 
Tsuyoshi Miezaki\thanks{Research Fellow of the Japan Society for 
the Promotion of Science and 
Department of Mathematics, Hokkaido University, 
Hokkaido 060--0810, Japan. email: miezaki@math.sci.hokudai.ac.jp
}
}

\maketitle

\begin{abstract}
In this paper, we give a new upper bound on the minimum 
Euclidean weight of Type~II $\ZZ_{2k}$-codes and 
the concept of extremality for the Euclidean weights
when $k=3,4,5,6$.
Together with the known result, we demonstrate that
there is an extremal Type~II $\ZZ_{2k}$-code of 
length $8m$ $(m \le 8)$ when $k=3,4,5,6$.
\end{abstract}

{\small

\quad
}

\section{Introduction}
Let $\ZZ_{2k}$ be the ring 
of integers modulo $2k$, where $k$ 
is a positive integer.
In this paper, 
we take the set of elements of $\ZZ_{2k}$ to be either 
$\{0,1,\ldots,2k-1\}$ or
$\{0,\pm 1,\ldots,\pm(k-1),k\}$.
A $\ZZ_{2k}$-code $C$ of length $n$
(or a code $C$ of length $n$ over $\ZZ_{2k}$)
is a $\ZZ_{2k}$-submodule of $\ZZ_{2k}^n$.
The Euclidean weight of a codeword 
$x=(x_1,x_2,\ldots,x_n)$ is 
$\sum_{i=1}^n \min\{x_i^2,(2k-x_i)^2\}$.
The minimum Euclidean weight $d_E(C)$ of $C$ is the smallest Euclidean
weight among all nonzero codewords of $C$.

A $\ZZ_{2k}$-code $C$ is {\em self-dual} if $C=C^\perp$ where
the dual code $C^\perp$ of $C$ is defined as 
$C^\perp = \{ x \in \ZZ_{2k}^n \mid x \cdot y = 0$ for all $y \in C\}$
under the standard inner product $x \cdot y$. 
As described in~\cite{RS-Handbook},
self-dual codes are an important class of linear codes for both
theoretical and practical reasons.
It is a fundamental problem to classify self-dual codes
of modest lengths 
and determine the largest minimum weight among self-dual codes
of that length.

A binary doubly even self-dual code is often called Type~II.
For $\ZZ_4$-codes, Type~II codes were first defined 
in~\cite{Z4-BSBM} as self-dual codes containing a 
$(\pm 1)$-vector and with the property that all
Euclidean weights are divisible by eight.
Then it was shown in~\cite{Z4-HSG} that,
more generally, the condition of containing a 
$(\pm 1)$-vector is redundant.
For general $k$,
Type~II $\ZZ_{2k}$-codes were defined in~\cite{BDHO} 
as a self-dual code with the property that all
Euclidean weights are divisible by $4k$.
It is known that a Type~II $\ZZ_{2k}$-code of length $n$ exists 
if and only if $n$ is divisible by eight~\cite{BDHO}.

The aim of this paper is to show the following theorem.

\begin{thm}\label{main1}
Let $C$ be a Type~II $\ZZ_{2k}$-code of length $n$.
If $k \le 6$
then the minimum Euclidean weight $d_E(C)$ of $C$ 
is bounded by
\begin{equation}\label{Eq:B}
d_E(C) \le 4k \Big\lfloor \frac{n}{24} \Big\rfloor +4k.
\end{equation}
\end{thm}

\begin{rem}
The upper bound (\ref{Eq:B})
is  known for the cases $k=1$~\cite{MS73}
and $k=2$~\cite{Z4-BSBM}.
For $k \ge 3$, the bound (\ref{Eq:B}) is known under the
assumption that $\lfloor n/24 \rfloor\leq k-2$~\cite{BDHO}.
\end{rem}

We say that a Type~II $\ZZ_{2k}$-code meeting the bound (\ref{Eq:B})
with equality is {\em extremal} for $k \le 6$.
For the following cases
\[
(k,m)=(4,7),(4,8),(5,6), (5,7),(5,8),(6,4),(6,5),(6,6),
(6,7) \text{ and }(6,8),
\]
an extremal Type~II $\ZZ_{2k}$-code of length $8m$
is constructed for the first time in Section \ref{Sec:code}.
Together with the known results on the existences of extremal
Type II codes,
we have the following theorem.

\begin{thm}\label{Thm:E}
If $k \le 6$ then
there is an extremal Type~II $\ZZ_{2k}$-code of 
length $8m$ for $m \le 8$.
\end{thm}

The existences of a binary extremal Type~II code of length $72$ and a
$72$-dimensional extremal even unimodular (Type~II) lattice
are long-standing open questions.
In this paper, we propose the following new question.

\bigskip\noindent
{\bf Question.}
Is there an extremal Type~II $\ZZ_{2k}$-code of length $72$
for $k \le 6$?
\bigskip

We remark that if there is an Type~II $\ZZ_{2k}$-code of length $72$
$(k=4,5,6)$ then a $72$-dimensional  extremal even unimodular lattice can be obtained
by Construction A.

All computer calculations in this paper
were done by {\sc Magma}~\cite{Magma}.

\section{Preliminaries}

An $n$-dimensional (Euclidean) lattice $\Lambda$ is a subset of
$\RR^{n}$ with the property that there exists a basis 
$\{e_{1}, e_2,\ldots, e_{n}\}$ of $\RR^{n}$ such that 
$\Lambda =\ZZ e_{1}\oplus \ZZ e_{2}\oplus \cdots \oplus\ZZ e_{n}$, i.e., 
$\Lambda$ consists of all integral linear combinations of the 
vectors $e_{1}, e_{2}, \ldots, e_{n}$. 
The dual lattice $\Lambda^*$ of $\Lambda$ is the lattice
$
\{x\in \RR^{n} \mid \langle x,y \rangle \in\ZZ \text{ for all }
y\in \Lambda\}
$
where $\langle x, y \rangle$ is the standard inner product.
A lattice with $\Lambda=\Lambda^{*}$
is called {\em unimodular}.
The norm of a vector $x$ is $\langle x, x \rangle$.
A unimodular lattice with even norms is said to be {\em even}.
A unimodular lattice containing a vector of odd norm is said to be
{\em odd}.
An $n$-dimensional even unimodular lattice exists if and only
if $n \equiv 0 \pmod 8$ while an odd unimodular lattice
exists for every dimension.
The minimum norm $\min(\Lambda)$ of $\Lambda$ is the smallest
norm among all nonzero vectors of $\Lambda$.
For $\Lambda$ and a positive integer $m$, 
the shell $\Lambda_m$ of norm $m$ is defined as
$\{x\in \Lambda \mid \langle x,x \rangle=m \}$.
Two lattices $L$ and $L'$ are {\em isomorphic}, denoted $L \simeq L'$,
if there exists an orthogonal matrix $A$ with
$L' = L \cdot A=\{xA \mid x \in L\}$.
Two lattices $L$ and $L'$ are {\em neighbors} if
both lattices contain a sublattice of index $2$
in common.

The theta series $\Theta_{\Lambda}(q)$ of $\Lambda$ is the 
following formal power series
\[
\Theta_{\Lambda}(q)=
\sum_{x\in \Lambda}q^{\langle x, x \rangle}
=\sum_{m=0}^{\scriptstyle \infty}\vert \Lambda_{m}\vert q^{m}.
\]
For example, when $\Lambda$ is the $E_{8}$-lattice
\begin{align*}
\Theta_{\Lambda}(q)=E_{4}(q)&=1+240\sum_{m=1}^{\infty}\sigma _{3}(m) q^{2m} \\
&=1 + 240 q^2 + 2160 q^4 + 6720 q^6 +\cdots, 
\end{align*}
where $\sigma _{3}(m)$ is a divisor function $\sigma _{3}(m)=\sum_{0<d|m}d^3$. 
Moreover the following theorem 
is known (see~\cite[Chap.~7]{SPLAG}).


\begin{thm}
\label{thm:Hecke}
If $\Lambda$ is an even unimodular lattice then 
\[
\Theta _{\Lambda}(q)\in \CC[E_{4}(q), \Delta _{24}(q)],
\]
where 
$\Delta _{24}(q)=q^{2}\prod _{m=1}^{\scriptstyle \infty}
(1-q^{2m})^{24}$. 
\end{thm}

We now give a method to construct 
even unimodular lattices from Type~II codes, which
is called Construction A~\cite{BDHO}.
Let $\rho$ be a map from $\ZZ_{2k}$ to $\ZZ$ sending $0, 1, \ldots , k$ 
to $0, 1, \ldots , k$ and $k+1, \ldots , 2k-1$ to $1-k, \ldots , -1$, 
respectively. 
If $C$ is a  self-dual $\ZZ_{2k}$-code of length $n$, then 
the lattice 
\[
A_{2k}(C)=\frac{1}{\sqrt{2k}}\{\rho (C) +2k \ZZ^{n}\} 
\]
is an $n$-dimensional unimodular lattice, where 
\[
\rho (C)=\{(\rho (c_{1}), \ldots , \rho (c_{n}))\ 
\vert\ (c_{1}, \ldots , c_{n}) \in C\}. 
\]
The minimum norm of $A_{2k}(C)$ is $\min\{2k, d_{E}(C)/2k\}$.
Moreover, 
if $C$ is Type~II then the lattice $A_{2k}(C)$ 
is an even unimodular lattice~\cite{BDHO}.

The symmetrized weight enumerator of a $\ZZ_{2k}$-code
$C$ is 
\begin{align*}
\swe_{C}(x_{0}, x_{1}, \ldots , x_{k})=
\sum_{c\in C}x_{0}^{n_{0}(c)}x_{1}^{n_{1}(c)}\cdots 
x_{k-1}^{n_{k-1}(c)}x_{k}^{n_{k}(c)}, 
\end{align*}
where $n_{0}(c)$, $n_{1}(c)$, \ldots , 
$n_{k-1}(c)$, $n_{k}(c)$ are the numbers of 
$0, \pm 1, \ldots , \pm (k-1), k$ components of $c$, respectively~\cite{BDHO}. 
Then the theta series of $A_{2k}(C)$ can be found by replacing 
$x_{1}$, $x_{2}$, $\ldots$, $x_{k}$ by 
\begin{align*}
f_{0}=\sum_{x\in 2k\ZZ}q^{x^{2}/2k}, 
f_{1}=\sum_{x\in 2k\ZZ +1}q^{x^{2}/2k}, \ldots ,
f_{k}=\sum_{x\in 2k\ZZ +k}q^{x^{2}/2k}, 
\end{align*}
respectively.

\section{Proof of Theorem~\ref{main1}}
In this section, we give a proof of Theorem~\ref{main1}.
Our proof is an analogue of that 
of~\cite[Corollary 13]{Z4-BSBM} (see also~\cite{MOS75}).
We remark that 
in the proof of~\cite[Corollary 13]{Z4-BSBM}
$(\Delta /E_8^34)$ (p.~973, right, l.~$-7$)
should be $(t E_{4}^{3}/\Delta )$ and
$(4\ZZ)^8/2$ (p.~973, right, l.~$-5$) should be
$2\ZZ^8$.

\begin{proof}
Let $C$ be a Type~II $\ZZ_{2k}$-code of length $n$. Then 
the even unimodular lattice $A_{2k}(C)$ contains the sublattice 
$\Lambda _{0}=\sqrt{2k}\ZZ^{n}$ which has minimum norm $2k$.
We set $\Theta _{\Lambda_{0}}(q)=\theta _{0}$,
$n=8j$ and $j=3\mu +\nu$ ($\nu=0,1,2$), that is, $\mu=\lfloor n/24 \rfloor$.
In this proof, we denote $E_4(q)$ and $\Delta_{24}(q)$ by
$E_{4}$ and $\Delta$, respectively.
By Theorem~\ref{thm:Hecke}, 
the theta series of $A_{2k}(C)$ can be written as 
\begin{align*}
\Theta _{A_{2k}(C)}(q)=\sum_{s=0}^{\mu}a_{s}E_{4}^{j-3s}\Delta ^{s}
=\sum_{r\geq 0}\vert {A_{2k}(C)}_{r}\vert q^{r}
=\theta_{0}+\sum_{r\geq 1}\beta _{r}q^{r}. 
\end{align*}

Suppose that $d_{E}(C)\geq 4k(\mu +1)$. 
We remark that a codeword of Euclidean weight $4km$
gives a vector of norm $2m$ in $A_{2k}(C)$. 
Then we choose the $a_{0}, a_1, \ldots , a_{\mu}$ so that 
\begin{align*}
\Theta _{A_{2k}(C)}(q)=\theta_{0}+\sum_{r\geq 2(\mu+1)}\beta ^{\ast}_{r}q^{r}. 
\end{align*}

Here, we set $b_{2s}$ as 
$E_{4}^{-j}\theta_{0}=\sum_{s=0}^{\infty}b_{2s}(\Delta/E_{4}^{3})^s$. 
That is, $\theta_{0}=\sum_{s=0}^{\infty}b_{2s}E_{4}^{j-3s}\Delta^s$. 
Then 
\begin{align*}
\sum_{s=0}^{\mu}a_{s}E_{4}^{j-3s}\Delta ^{s}
=\Theta _{A_{2k}(C)}(q)
=\sum_{s=0}^{\infty}b_{2s}(\Delta/E_{4}^{3})^s
+\sum_{r\geq 2(\mu+1)}\beta ^{\ast}_{r}q^{r}. 
\end{align*}
Comparing the coefficients of $q^{i}$ $(0\leq i\leq 2\mu)$, we get 
$a_{s}=b_{2s}$ $(0\leq s\leq \mu)$. 
Hence we have 
\begin{align}
-\sum_{r\geq (\mu+1)}b_{2r}E_{4}^{j-3r}\Delta^r
=\sum_{r\geq 2(\mu+1)}\beta^{\ast}_{r}q^{r}. \label{eqn:b}
\end{align}
In (\ref{eqn:b}), comparing the coefficient of $q^{2(\mu+1)}$, 
we have
\[
\beta^{\ast}_{2(\mu +1)}=-b_{2(\mu +1)}. 
\]
All the series are in $q^{2}=t$, 
and B\"urman's formula shows that 
\[
b_{2s}=\frac{1}{s!}\frac{d^{s-1}}{dt^{s-1}}
\left(
\left(
\frac{d}{dt}(E_{4}^{-j}\theta _{0})\right)
(t E_{4}^{3}/\Delta )^{s}\right)_{\{t=0\}}. 
\]
Using the fact that $\theta _{0}=\theta_{1}^{j}$ where $\theta _{1}$ 
is the theta series of the lattice $\sqrt{2k}\ZZ^{8}$,
\begin{align*}
b_{2s}=\frac{-j}{s!}\frac{d^{s-1}}{dt^{s-1}}\left(
E_{4}^{3s-j-1}\theta _{1}^{j-1}(\theta_{1}E_{4}^{\prime}
-\theta_{1}^{\prime}E_{4})(t/\Delta)^{s}\right)_{\{t=0\}}, 
\end{align*}
where $f^{\prime}$ is the derivation of $f$ with respect to $t=q^{2}$. 

The condition that there is a codeword of Euclidean weight 
$4k(\mu +1)$ is equivalent to the condition 
$\beta^{\ast}_{2(\mu+1)}>0$. 
It is sufficient to show that the coefficients of 
$\theta_{1}^{j-1}(\theta_{1}E_{4}^{\prime}-\theta_{1}^{\prime}E_{4})$ 
are positive up to the exponent $\mu$ since $E_{4}$ and $1/\Delta$ 
have positive coefficients. 

By Proposition 3.4 in~\cite{BDHO},
there exists a Type~II $\ZZ_{2k}$-code of length $8$
for every $k$.
Hence let $C_{8}$ be a Type~II $\ZZ_{2k}$-code of length $8$. 
Then $A_{2k}(C_8)$ is the $E_{8}$-lattice. 
In addition,  we can write 
\begin{align*}
E_{4}=\swe_{C_{8}}(f_{0}, f_{1}, \ldots , f_{k})\ 
{\rm and}\ \theta _{1}=f_{0}^{8}. 
\end{align*}
Deriving 
\begin{align*}
E_{4}/{\theta_{1}}=\swe_{C_{8}}(1, f_{1}/f_{0}, \ldots , f_{k}/f_{0}), 
\end{align*}
we find 
\begin{align*}
\theta_{1}^{j-1}&(\theta_{1}E_{4}^{\prime}-\theta_{1}^{\prime}E_{4})\\
=&\frac{\partial \swe_{C_{8}}(f_{0}, f_{1}, \ldots , f_{k})}
{\partial x_{1}}
f_{0}^{8j-1}
(f_{0}f_{1}^{\prime}-f_{0}^{\prime}f_{1}) \\
&+\cdots +
\frac{\partial \swe_{C_{8}}(f_{0}, f_{1}, \ldots , f_{k})}
{\partial x_{k}}f_{0}^{8j-1}
(f_{0}f_{k}^{\prime}-f_{0}^{\prime}f_{k}). 
\end{align*}
Hence it is sufficient to show that $f_{0}^{8j-1}
(f_{0}f_{1}^{\prime}-f_{0}^{\prime}f_{1}), \ldots , 
f_{0}^{8j-1}
(f_{0}f_{k}^{\prime}-f_{0}^{\prime}f_{k})$ 
have positive coefficients up to $\mu$. 
We only consider the case $f_{0}^{8j-1}
(f_{0}f_{1}^{\prime}-f_{0}^{\prime}f_{1})$ 
and the other cases are similar. 
We have that
\begin{align*}
t(f_{0}f_{1}^{\prime}&-f_{0}^{\prime}f_{1})
=\sum_{x, y\in \ZZ}\frac{(1+2k y)^{2}-(2k x)^{2}}{4k}t^{((1+2k y)^{2}+(2k x)^{2})/4k},
\end{align*}
then
\begin{multline}\label{Eq:2}
tf_{0}^{s}(f_{0}f_{1}^{\prime}-f_{0}^{\prime}f_{1}) \\ 
=\sum_{x, y, x_{1}, \ldots , x_{s}\in \ZZ}
\frac{(1+2k y)^{2}-(2k x)^{2}}{4k}
\cdot t^{((1+2k y)^{2}+(2k x)^{2}+(2k x_{1})^{2}+\cdots +(2k x_{s})^{2})/4k}. 
\end{multline}
Fix one of the choices
$y$, $x$, $x_1,\ldots x_s \in \ZZ$ and define $l$
as follows:
\begin{align}\label{Eq:2-2}
l=(1+2k y)^{2}+(2k x)^{2}+(2k x_{1})^{2}+\cdots +(2k x_{s})^{2}.
\end{align}
Consider all permutations on the set \{$x$, $x_1, \ldots$, $x_s$\}.
As the sum of coefficients of $t^{l/4k}$ in the
right hand side of (\ref{Eq:2}) under these cases,
we have that some positive constant multiple by 
\begin{multline}\label{Eq:3}
\frac{(s+1)(1+2k y)^{2}-(2k x)^{2}-(2k x_{1})^{2}-\cdots -(2k x_{s})^{2}}{4k}
\\
=\frac{(s+2)(1+2k y)^{2}-l}{4k}.
\end{multline}
If $l < s +2$ then (\ref{Eq:3}) is positive.
Since we consider the case $s=8j-1$,
$l < n+1$.
Hence if the exponent $l/4k$ of $t$ is less than
$(n+1)/4k$ then (\ref{Eq:3}) is positive.
This means that if $\mu < (n+1)/4k$ then (\ref{Eq:3}) is positive.
This condition $\mu < (n+1)/4k$
is satisfied since $k\leq 6$.
Thus for any choice $y$, $x$, $x_1,\ldots x_s$,
(\ref{Eq:3}) is positive.
The coefficient of $t^{l/4k}$ in the
right hand side of (\ref{Eq:2})
is the sum of those coefficients (\ref{Eq:3}), that is, positive.
This completes the proof of Theorem~\ref{main1}.


\end{proof}

\section{Extremal Type~II $\ZZ_{2k}$-codes}
\label{Sec:code}
An extremal Type~II $\ZZ_{2k}$-code of length $8m$ 
is currently known for the cases $(k,m)$ listed in the
second column in 
Table~\ref{Tab:Exi}.
In this section,
an extremal Type~II $\ZZ_{2k}$-code of length $8m$
is constructed for the first time
for the cases $(k,m)$ listed in the
last column in 
Table~\ref{Tab:Exi}.

\begin{table}[thb]
\caption{Existence of extremal Type~II $\ZZ_{2k}$-codes of length $8m$}
\label{Tab:Exi}
\begin{center}
{\footnotesize
\begin{tabular}{c|ll|ll}
\noalign{\hrule height0.8pt}
$k$ & \multicolumn{2}{c|}{$m$ (known cases)}
&\multicolumn{2}{c}{$m$ (new cases)}\\
\hline
$1$ & $1,2,\ldots,8,10,11,13,14,17$ 
&\cite[p.~194]{SPLAG}, \cite{H112}   & \\
$2$ & $1,2,\ldots,8$ &\cite{Z4-BSBM}, \cite{Cal-S}, \cite{H-Z4-56-64} & \\
$3$ & $1,2,\ldots,8$ &\cite{Chapman}, \cite{HKO}   & \\
$4$ & $1,2,\ldots,6$ &\cite{Chapman}, \cite{GH05}  
& 7, 8 & ($C_{8,56}$, $C_{8,64}$) \\
$5$ & $1,2,\ldots,5$ &\cite{Chapman}, \cite{GH05}  
& 6, 7, 8 & (Proposition \ref{Prop:Z10}, $C_{10,56}$, $C_{10,64}$) \\
$6$ & $1,2,3$          &\cite{Chapman}               
& 4, 5, 6, 7, 8 & (Proposition \ref{Prop:Z12})\\
\noalign{\hrule height0.8pt}
   \end{tabular}
}
\end{center}
\end{table}


Let $A$ and $B$ be $n \times n$ negacirculant matrices,
that is, $A$ and $B$ have the following form
\[
\left( \begin{array}{ccccc}
r_0&r_1&r_2& \cdots &r_{n-1} \\
-r_{n-1}&r_0&r_1& \cdots &r_{n-2} \\
-r_{n-2}&-r_{n-1}&r_0& \cdots &r_{n-3} \\
\vdots &\vdots & \vdots && \vdots\\
-r_1&-r_2&-r_3& \cdots&r_0
\end{array}
\right).
\]
If  $AA^T+BB^T=-I_n$, then it is trivial that
\begin{equation} \label{Eq:GM}
\left(
\begin{array}{ccc@{}c}
\quad & {\Large I_{2n}} & \quad &
\begin{array}{cc}
A & B \\
-B^T & A^T
\end{array}
\end{array}
\right)
\end{equation}
generates a self-dual code
where $I_{n}$ denotes the identity matrix of order $n$
and  $A^T$ is the transpose of $A$.

\begin{table}[thb]
\caption{New extremal Type~II $\ZZ_{2k}$-codes}
\label{Tab:GM}
\begin{center}
{\footnotesize
\begin{tabular}{c|l|l}
\noalign{\hrule height0.8pt}
Codes & \multicolumn{1}{c|}{$r_A$}&\multicolumn{1}{c}{$r_B$}\\
\hline
$C_{8,56}$ & 
$(0,0,4,3,4,1,6,3,1,1,1,1,1,2)$&$(1,1,0,0,0,0,0,3,2,0,0,0,0,4)$\\
$C_{8,64}$ & 
$(0,0,0,2,0,7,3,2,0,0,5,3,1,4,0,2)$&$(0,0,1,0,0,0,0,1,7,1,3,0,1,2,2,0)$\\
\hline
$C_{10,56}$ &
$(0,0,0,2,5,1,4,1,2,0,5,0,4,1)$&$(0,0,0,0,0,1,0,5,7,0,3,9,1,0)$ \\
$C_{10,64}$ &
$(0,0,4,3,2,0,0,1,9,0,0,0,9,1,2,0)$&$(1,3,0,1,0,6,9,4,6,2,0,5,0,0,2,3)$\\
\hline
$C_{12,32}$ &$(0,0,7,6,0,1,7,10)$&$(0,1,0,4,1,0,3,11)$\\
$C_{12,40}$ &$(0,0,0,2,1,10,5,9,2,10)$&$(0,1,0,1,0,0,11,1,0,4)$\\
$C_{12,56}$ &
$(2,11,2,2,4,11,0,5,0,0,6,1,5,7)$&$(1,0,5,3,0,8,0,2,0,7,7,0,0,4)$\\
\noalign{\hrule height0.8pt}
   \end{tabular}
}
\end{center}
\end{table}


Using the above construction method, we have found
extremal Type~II $\ZZ_{8}$-codes $C_{8,56}$ and $C_{8,64}$
of lengths $56$ and $64$, respectively, and
extremal Type~II $\ZZ_{10}$-codes $C_{10,56}$ 
and $C_{10,64}$ of lengths $56$ and $64$, respectively.
The first rows $r_A$ and $r_B$ of the matrices $A$ and $B$
in their generator matrices (\ref{Eq:GM}) are listed in 
Table~\ref{Tab:GM}.
Hence we have the following:

\begin{prop}\label{Prop:5664}
For lengths $56$ and $64$,
there is an extremal Type~II $\ZZ_{2k}$-code when
$k=4$ and $5$.
\end{prop}

An $n$-dimensional even unimodular lattice is
called {\em extremal} if it has minimum norm
$2\lfloor n/24\rfloor+2$.
The existence of an extremal Type~II $\ZZ_{10}$-code
of length $48$ is established
by considering the existence of a $10$-frame in
some extremal even unimodular lattice.
Recall that 
a set $\{f_1, \ldots, f_{n}\}$ of $n$ vectors $f_1, \ldots, f_{n}$ in an
$n$-dimensional unimodular lattice $L$ with
$ \langle f_i, f_j \rangle = \ell \delta_{i,j}$
is called an {\em $\ell$-frame} of $L$,
where $\delta_{i,j}$ is the Kronecker delta.
It is known that an even unimodular lattice $L$
contains a $2k$-frame if and only if there is 
a Type~II $\ZZ_{2k}$-code $C$ 
such that $A_{2k}(C) \simeq L$.

\begin{prop}\label{Prop:Z10}
There is an extremal Type~II $\ZZ_{10}$-code of length $48$.
\end{prop}
\begin{proof}
Let $C_{5,48}$ be the $\FF_5$-code with generator 
matrix (\ref{Eq:GM}) where the first rows $r_A$ and $r_B$ of 
the matrices $A$ and $B$ are
\[
r_A=(2,3,0,2,2,3,2,2,3,2,2,0) \text{ and }
r_B=(3,0,4,4,0,1,0,0,4,0,0,1),
\]
respectively.
Then this code $C_{5,48}$ is a self-dual code and
the lattice $A_5(C_{5,48})=\frac{1}{\sqrt{5}}\{x \in \ZZ^{48}
\mid x \pmod 5 \in C_{5,48}\}$ is an
odd unimodular lattice.
The lattice has theta series
$1 + 393216 q^5 + 26201600q^6 + \cdots$.
We have verified that 
$A_5(C_{5,48})$ has an even unimodular neighbor $L_{48}$
which is extremal.
Clearly the lattice $A_5(C_{5,48})$ contains 
the $5$-frame
$\{\sqrt{5}e_1,\sqrt{5}e_2,\ldots,\sqrt{5}e_{48}\}$
where $e_i$ ($i=1,2,\ldots,{48}$) denotes the $i$-th unit vector
$(\delta_{i,1},\delta_{i,2},\ldots,\delta_{i,48})$
of length $48$.
Then the set
$F=\{\sqrt{5}(e_{2i-1} \pm e_{2i}) \mid i=1,2,\ldots,24 \}$
is a $10$-frame of the even sublattice of $A_5(C_{5,48})$.
Hence $F$ is also a $10$-frame of the extremal
even unimodular neighbor $L_{48}$.
Therefore there is a Type~II $\ZZ_{10}$-code $C_{10,48}$ of
length $48$ such that 
$A_{10}(C_{10,48}) \simeq L_{48}$.
Moreover, the code $C_{10,48}$ must be extremal since 
the lattice $L_{48}$ is extremal.
\end{proof}

Similar to the above proposition, 
the existence of $12$-frames in
extremal even unimodular lattices yields that 
of some extremal Type~II $\ZZ_{12}$-codes.

\begin{prop}\label{Prop:Z12}
There is an extremal Type~II $\ZZ_{12}$-code of length
$8m$ for $m=4,5,6,7$ and $8$.
\end{prop}
\begin{proof}
It is known that 
there is an extremal Type~II $\ZZ_{6}$-code of length
$8m$ for $m=4,5,6,7$ and $8$ 
(see Table \ref{Tab:Exi}).
We denote these codes by $C_{6,8m}$ $(m=4,5,6,7 \text{ and } 8)$, 
respectively.
Since $C_{6,8m}$ is an extremal Type II code,
the lattice $A_6(C_{6,8m})$ is an extremal 
even unimodular lattice
for $m=4,5,6,7 \text{ and } 8$.
Moreover, clearly the lattice $A_6(C_{6,8m})$ contains 
the $6$-frame
$\{\sqrt{6}e_1,\sqrt{6}e_2,\ldots,\sqrt{6}e_{8m}\}$
where $e_i$ denotes the $i$-th unit vector of length $8m$.
Then the set
$\{\sqrt{6}(e_{2i-1} \pm e_{2i}) \mid i=1,2,\ldots,4m \}$
is a $12$-frame of $A_6(C_{6,8m})$.
Hence there is a Type~II $\ZZ_{12}$-code $N_{8m}$ of
length $8m$ such that 
$A_{12}(N_{8m}) \simeq A_6(C_{6,8m})$.
Moreover, the code $N_{8m}$ must be extremal since 
the lattice $A_6(C_{6,8m})$ is extremal.
\end{proof}
\begin{rem}
Similar to Proposition \ref{Prop:5664},
by considering generator matrices (\ref{Eq:GM}),
we have found extremal Type~II $\ZZ_{12}$-codes 
$C_{12,32}$, $C_{12,40}$ and $C_{12,56}$ 
of lengths $32,40$ and $56$, respectively where 
the first rows $r_A$ and $r_B$ of the matrices $A$ and $B$
in  (\ref{Eq:GM}) are listed
in 
Table~\ref{Tab:GM}.
\end{rem}

Together with the known results on the existences of extremal
Type II codes 
(see Table \ref{Tab:Exi}), 
Propositions \ref{Prop:5664}, \ref{Prop:Z10} and \ref{Prop:Z12}
give Theorem \ref{Thm:E}.

\bigskip
\noindent
{\bf Acknowledgment.}
The authors would like to thank Kenichiro Tanabe for
useful discussions, and 
Koichi Betsumiya
for helpful conversations on the construction of
extremal Type~II codes.


\end{document}